\documentclass[11pt]{article}
\usepackage{amsmath,amsfonts,amssymb,amsthm}
\usepackage{bbm}
\usepackage{enumerate}
\usepackage{esint}
\usepackage[pagebackref,colorlinks,citecolor=blue,linkcolor=blue]{hyperref}
\usepackage[dvipsnames]{xcolor}
\usepackage{mathtools} 

\usepackage{comment}

\usepackage{geometry}
\numberwithin{equation}{section}
\theoremstyle{plain}
\newtheorem{theorem}[equation]{Theorem}
\newtheorem{corollary}[equation]{Corollary}
\newtheorem{lemma}[equation]{Lemma}
\newtheorem{proposition}[equation]{Proposition}

\theoremstyle{definition}
\newtheorem{definition}[equation]{Definition}

\newtheorem{remark}[equation]{Remark}

\numberwithin{equation}{section}

\newcommand{\R}{{\mathbb R}}
\newcommand{\N}{{\mathbb N}}

\newcommand{\Om}{\Omega}

\providecommand{\vint}[1]{\mathchoice
	{\mathop{\vrule width 5pt height 3 pt depth -2.5pt
			\kern -9pt \kern 1pt\intop}\nolimits_{\kern -5pt{#1}}}
	{\mathop{\vrule width 5pt height 3 pt depth -2.6pt
			\kern -6pt \intop}\nolimits_{\kern -3pt{#1}}}
	{\mathop{\vrule width 5pt height 3 pt depth -2.6pt
			\kern -6pt \intop}\nolimits_{\kern -3pt{#1}}}
	{\mathop{\vrule width 5pt height 3 pt depth -2.6pt
			\kern -6pt \intop}\nolimits_{\kern -3pt{#1}}}}

\newcommand{\eps}{\varepsilon}
\newcommand{\loc}{\mathrm{loc}}

\newcommand{\BV}{\mathrm{BV}}
\newcommand{\SBV}{\mathrm{SBV}}

\newcommand{\ch}{\text{\raise 1.3pt \hbox{$\chi$}\kern-0.2pt}}

\newcommand{\mres}{\mathbin{\vrule height 2ex depth 2.2pt width
		0.12ex\vrule height -0.3ex depth 2.2pt width .5ex}}

\DeclareMathOperator{\dive}{div}
\DeclareMathOperator{\dist}{dist}

\DeclareMathOperator{\Var}{Var}

\DeclareMathOperator{\BMO}{BMO}

\begin{document}
	\title{BMO-type functionals, total variation,\\
		and $\Gamma$-convergence
		\footnote{
			\hfill \break {\it Keywords\,}: Total variation, special function of bounded variation,
			nonlocal functional, $\Gamma$-convergence, bounded mean oscillation
	}}
	\author{Panu Lahti and Quoc-Hung Nguyen}

	\maketitle
	
	\begin{abstract}
		We study the BMO-type functional $\kappa_{\eps}(f,\R^n)$, which can be used to characterize
		BV functions $f\in\BV(\R^n)$.
		The $\Gamma$-limit of this functional, taken with respect to $L^1_{\loc}$-convergence, is
		known to be $\tfrac 14 |Df|(\R^n)$. We show that the $\Gamma$-limit with respect to
		$L^{\infty}_{\loc}$-convergence is
		\[
		\tfrac 14 |D^a f|(\R^n)+\tfrac 14 |D^c f|(\R^n)+\tfrac 12 |D^j f|(\R^n),
		\]
		which agrees with the ``pointwise'' limit in the case of SBV functions. 
	\end{abstract}
	
	\section{Introduction}
	
	In the past two decades,
	there has been widespread interest in characterizing Sobolev and BV (bounded variation)
	functions by means of non-local functionals;
	see e.g. \cite{BBM,Bre,BN16,Bre2,Dav}.
	One of these is the following BMO-type functional:
	given an open set $\Om\subset \R^n$ and a function $f\in L^1_{\loc}(\Om)$, one defines
	\[
	\kappa_{\eps}(f,\Om)
	:=\eps^{n-1}\sup_{\mathcal F_{\eps}}\sum_{Q\in \mathcal F_{\eps}}\vint{Q}\left|f(x)
	-\vint Q f(y)\,dy\right|
	\,dx,
	\]
	where $\mathcal F_{\eps}$ is a family of pairwise
	disjoint open cubes contained in $\Om$, with side length $\eps>0$ and arbitrary orientation.
	A variant of this was first studied by Bourgain--Brezis--Mironescu \cite{BBM15},
	and afterwards Ambrosio--Bourgain--Brezis--Figalli \cite{ABBF}
	used the functional $\kappa_{\eps}$ to characterize sets of finite perimeter.
	The functional $\kappa_{\eps}$ was studied in the context of more general BV functions,
	in particular special functions of bounded variation (SBV),
	by Ponce--Spector \cite{PoSp}, Fusco--Moscariello--Sbordone \cite{FMS16}, and
	De Philippis--Fusco--Pratelli \cite{DPFP}.
	We will give definitions in Section \ref{sec:definitions}.
	In particular, in \cite[Corollary 6.2]{DPFP} it was shown that
	for $f\in \SBV_{\loc}(\Om)$, we have
	\begin{equation}\label{eq:SBV limit}
		\lim_{\eps\to 0}\kappa_{\eps}(f,\Om)=\tfrac 14 |D^a f|(\Om)+\tfrac 12 |D^j f|(\Om).
	\end{equation}
	However, for a general BV function $f$,
	the limit $\lim_{\eps\to 0}\kappa_{\eps}(f,\Om)$ does not always exist.
	De Philippis--Fusco--Pratelli \cite[Example 6.3]{DPFP} give an example of
	$f\in\BV(\Om)$
	for which the variation measure $Df$ only consists of the Cantor part $D^c f$, with
	\[
	\liminf_{\eps\to 0}\kappa_{\eps}(f,\Om)=\kappa_-
	\quad\textrm{and}\quad \limsup_{\eps\to 0}\kappa_{\eps}(f,\Om)=\kappa_+,
	\]
	where the numbers $1/4<\kappa_-<\kappa_+<1/2$ can be arranged so that
	$\kappa_-$ is arbitrarily close to $1/4$ and $\kappa_+$ is arbitrarily close to $1/2$.
	
	In the study of nonlocal functionals, this type of defect is rather common, and one
	can often fix the issue by considering $\Gamma$-convergence instead.
	See e.g. \cite{BN18,BN16-2} for this type of situation in the context of different functionals.
	
	\begin{definition}
		Let $\Om\subset \R^n$ be open.
		Functionals $\Phi_{\eps}\colon L^1_{\loc}(\Om)\to [0,\infty]$
		are said to $\Gamma$-converge to a functional $\Phi$
		as $\eps\to 0$
		if the following two properties hold:
		\begin{enumerate}
			\item For every $f\in L_{\loc}^1(\Om)$ and for every family $\{f_\eps\}_{\eps>0}$ such that
			$f_\eps\to f$ in $L^{1}_{\loc}(\Om)$ as $\eps\to 0$, we have
			\[
			\liminf_{\eps\to 0}\Phi_{\eps}(f_{\eps})\ge \Phi(f).
			\]
			\item For every $f\in L_{\loc}^1(\Om)$, there exists a family $\{f_\eps\}_{\eps>0}$ such that
			$f_{\eps}\to f$ in $L_{\loc}^1(\Om)$ as $\eps\to 0$, and
			\[
			\limsup_{\eps\to 0}\Phi_{\eps}(f_{\eps})\le \Phi(f).
			\]
		\end{enumerate}
	\end{definition}
	
	Arroyo-Rabasa--Bonicatto--Del Nin \cite{ARBDN} show that for the
	functionals $\kappa(\cdot,\R^n)\colon L^1_{\loc}(\R^n)\to [0,\infty]$, the $\Gamma$-limit is
	\[
	\tfrac 14 |Df|(\R^n),\quad f\in L^1_{\loc}(\R^n).
	\]
	In this way, by considering $\Gamma$-convergence we find that a limit exists.
	But on the downside, in the case of SBV functions
	the limit does not agree
	with \eqref{eq:SBV limit}.
	For this reason, we consider $\Gamma$-convergence with respect to more general (stronger, to be precise) topologies.
	\begin{definition}
		Let $\Om\subset \R^n$ be open and let $1\le p\le \infty$.
		Functionals $\Phi_{\eps}\colon L^1_{\loc}(\Om)\to [0,\infty]$
		are said to $\Gamma_{p}$-converge to a functional $\Phi$
		as $\eps\to 0$
		if the following two properties hold:
		\begin{enumerate}
			\item For every $f\in L_{\loc}^1(\Om)$ and for every family $\{f_\eps\}_{\eps>0}$ such that $f_\eps- f\in L^{p}_{\loc}(\Om)$ and 
			$f_\eps- f\to 0$ in $L^{p}_{\loc}(\Om)$ as $\eps\to 0$, we have
			\[
			\liminf_{\eps\to 0}\Phi_{\eps}(f_{\eps})\ge \Phi(f).
			\]
			\item For every $f\in L_{\loc}^1(\Om)$, there exists a family $\{f_\eps\}_{\eps>0}$ such that
		 $f_\eps- f\in L^{p}_{\loc}(\Om)$ and 
		$f_\eps- f\to 0$ in $L^{p}_{\loc}(\Om)$ as $\eps\to 0$, and
			\[
			\limsup_{\eps\to 0}\Phi_{\eps}(f_{\eps})\le \Phi(f).
			\]
		\end{enumerate}
	\end{definition}
	
	Our main result is the following.
	
	\begin{theorem}\label{thm:main}
		Let $\Om\subset \R^n$ be an open set. Then for $1\le p<\infty$, the $\Gamma_p$-limit
		of the functionals $\kappa_{\eps}(\cdot,\Om)$ is
		\begin{equation}\label{eq:Gamma p limit}
			\tfrac 14 |D f|(\Om),\quad f\in L^1_{\loc}(\Om),
		\end{equation}
		whereas the $\Gamma_\infty$-limit is
		\begin{equation}\label{eq:Gamma infty limit}
			\tfrac 14 |D^a f|(\Om)+\tfrac 14 |D^c f|(\Om)+\tfrac 12 |D^j f|(\Om),\quad f\in L^1_{\loc}(\Om).
		\end{equation}
	\end{theorem}
	
	Thus we see that for all $L^p_{\loc}$-topologies when $1\le p<\infty$, the $\Gamma_p$-limit is the same.
	However, when the topology is strengthened to $L^{\infty}_{\loc}$, we obtain a different limit,
	which coincides with \eqref{eq:SBV limit} in the case of SBV functions.
	
	\section{Preliminaries}\label{sec:definitions}
	
	Our definitions and notation are standard, and
	the reader may consult e.g. the monograph Evans--Gariepy \cite{EvGa} for more background.
	We will work in the Euclidean space $\R^n$ with $n\ge 1$.
	We denote the $n$-dimensional Lebesgue outer measure by $\mathcal L^n$.
	We denote the $s$-dimensional Hausdorff measure by $\mathcal H^{s}$;
	we will only consider $\mathcal H^{n-1}$.
	
	We denote the characteristic function of a set $A\subset\R^n$ by $\mathbbm{1}_A\colon \R^n\to \{0,1\}$.
	We denote by $|v|$ the Euclidean norm of $v\in \R^n$,
	and we also write $|A|:=\mathcal L^n(A)$ for a set $A\subset \R^n$.
	We write $B(x,r)$ for an open ball in $\R^n$ with center $x\in \R^n$
	and radius $r>0$, that is, $B(x,r)=\{y \in \R^n \colon |y-x|<r\}$.
	By $Q(x,r)$ we mean a cube centered at $x\in \R^n$, with side length $r$,
	and with sides parallel to the coordinate axes.
	We always understand cubes to be open.
	By $Q^{n-1}(z,r)$ we mean a cube
	in $\R^{n-1}$ centered at $z\in \R^{n-1}$, with side length $r$, and with sides parallel to the coordinate axes.
	
	By ``measurable'' we mean $\mathcal L^n$-measurable, unless otherwise specified.
	If a function $f$ is in $L^1(D)$ for some measurable set $D \subset \R^n$ of nonzero and finite Lebesgue
	measure, we write
	\[
	f_D:=\vint{D} f(y) \,d\mathcal L^n(y) \coloneqq \frac{1}{\mathcal L^n(D)} \int_D f(y) \,d\mathcal L^n(y)
	\] 
	for its mean value in $D$.
	
	By $\Om$ we always denote an open subset of $\R^n$.
	A sequence of positive Radon measures $\mu_i$ in $\Om$
	converges in the weak* sense to a positive Radon
	measure $\mu$, denoted by  $\mu_i\overset{*}{\rightharpoondown}\mu$, if
	\[
	\lim_{i\to\infty}\int_{\Om} \phi \,d\mu_i=\int_{\Om} \phi \,d\mu
	\quad \textrm{for all }\phi\in C_c(\Om).
	\]
	Let $\mu,\nu$ be two positive Radon measures on $\Om$.
	We can write the Radon--Nikodym decomposition of $\mu$ with respect to $\nu$ as
	\begin{equation}\label{eq:Radon Nikodym}
		\mu=\frac{d\mu}{d\nu}\,d\nu +\nu^s.
	\end{equation}
	By \cite[Theorem 3.2]{Fol} we have
	\begin{equation}\label{eq:diff wrt cubes}
		\frac{d\mu}{d\nu}(x)=\lim_{r\to 0}\frac{\mu(Q(x,r))}{\nu(Q(x,r))}
	\end{equation}
	for $\nu$-almost every $x\in\Om$.
	
	Let $S\subset \R^n$ be an $\mathcal H^{n-1}$-measurable set.
	We say that $S$ is countably $\mathcal H^{n-1}$-rectifiable if there exist countably many Lipschitz
	mappings $f_j\colon \R^{n-1}\to \R^n$ such that
	\[
	\mathcal H^{n-1}\left(S\setminus \bigcup_{j=1}^{\infty} f_j(\R^{n-1})\right)=0.
	\]
	
	The theory of $\BV$ mappings that we present next can be found in the monograph
	Ambrosio--Fusco--Pallara \cite{AFP}.
	As before, let $\Om\subset\R^n$ be an open set.
	A function
	$f\in L^1(\Omega)$ is of bounded variation,
	denoted $f\in \BV(\Omega)$, if its weak derivative
	is an $\R^{n}$-valued Radon measure with finite total variation. This means that
	there exists a (unique) Radon measure $Df$
	such that for all $\varphi\in C_c^1(\Omega)$, the integration-by-parts formula
	\[
	\int_{\Omega}f\frac{\partial\varphi}{\partial x_k}\,d\mathcal L^n
	=-\int_{\Omega}\varphi\,d(Df)_k,\quad k=1,\ldots,n,
	\]
	holds.
	The total variation of $Df$ is denoted by $|Df|$.
	If we do not know a priori that a mapping $f\in L^1_{\loc}(\Om)$
	is a BV function, we consider
	\[
	\Var(f,\Om):=\sup\left\{\int_{\Om}f_j\dive\varphi\,d\mathcal L^n,\,\varphi\in C_c^{1}(\Om;\R^{n}),
	\,|\varphi|\le 1\right\}.
	\]
	If $\Var(f,\Om)<\infty$, then the $\R^{n}$-valued Radon measure $Df$
	exists and $\Var(f,\Om)=|Df|(\Om)$
	by the Riesz representation theorem, and $f\in\BV(\Om)$ provided that $f\in L^1(\Om)$.
	
	The coarea formula states that for a function $f\in L^1_{\loc}(\Om)$ with $\Var(f,\Om)<\infty$, we have
	\begin{equation}\label{eq:coarea}
		|Df|(\Om)=\int_{-\infty}^{\infty}|D\mathbbm{1}_{\{f>t\}}|(\Om)\,dt.
	\end{equation}
	Here we abbreviate $\{f>t\}:=\{x\in \Om\colon f(x)>t\}$.
	
	Denote by $S_f\subset \Om$ the set of non-Lebesgue points of $f\in \BV_{\loc}(\Om)$.
	Given a unit vector $\nu\in \R^n$, we define the half-balls
	\begin{align*}
		B_{\nu}^+(x,r)\coloneqq \{y\in B(x,r)\colon( y-x)\cdot\nu>0\},\\
		B_{\nu}^-(x,r)\coloneqq \{y\in B(x,r)\colon (y-x)\cdot\nu<0\},
	\end{align*}
	where $( y-x)\cdot\nu$ denotes the inner product.
	We say that $x\in \Om$ is an approximate jump point of $f$ if there exist a unique unit vector $\nu\in \R^n$
	and distinct numbers $f^+(x), f^-(x)\in\R$ such that
	\[
	\lim_{r\to 0}\,\vint{B_{\nu}^+(x,r)}|f(y)-f^+(x)|\,d\mathcal L^n(y)=0
	\]
	and
	\[
	\lim_{r\to 0}\,\vint{B_{\nu}^-(x,r)}|f(y)-f^-(x)|\,d\mathcal L^n(y)=0.
	\]
	We define $\nu_f(x):=\nu.$
	The set of all approximate jump points is denoted by $J_f$.
	For $f\in\BV(\Om)$, we have that $\mathcal H^{n-1}(S_f\setminus J_f)=0$, see \cite[Theorem 3.78]{AFP},
	and also that $J_f$ is a countably $\mathcal H^{n-1}$-rectifiable Borel set.
	We write the Radon-Nikodym decomposition of the variation measure of $f$ into the absolutely continuous and singular parts with respect to $\mathcal L^n$
	as $Df=D^a f+D^s f$.
	Furthermore, we define the Cantor and jump parts of $Df$ by
	\[
	D^c f\coloneqq  D^s f\mres (\Om\setminus S_f),\qquad D^j f\coloneqq D^s f\mres J_f.
	\]
	Here
	\[
	D^s f \mres J_f(A):=D^s f (J_f\cap A),\quad \textrm{for } D^s f\textrm{-measurable } A\subset \R^n.
	\]
	Since $\mathcal H^{n-1}(S_f\setminus J_f)=0$ and $|Df|$ vanishes on
	$\mathcal H^{n-1}$-negligible sets, we get the decomposition (see \cite[Section 3.9]{AFP})
	\[
	Df=D^a f+ D^c f+ D^j f.
	\]
	When $|D^c f|(\Om)=0$, we say that $f$ is a special function of bounded variation, denoted by $f\in \SBV(\Om)$.
	When $|D^s f|(\Om)=0$, then $f$ is a Sobolev function, $f\in W^{1,1}(\Om)$.
	For the jump part, we know that
	\begin{equation}\label{eq:jump part representation}
		d|D^j f|=|f^{+}-f^-|\,d\mathcal H^{n-1}\mres J_f.
	\end{equation}

	If $\Var(f,\Om)=\infty$, we interpret $|Df|(\Om)=|D^a f|(\Om)=|D^c f|(\Om)=|D^j f|(\Om)=\infty$.
	
	\section{Lower bounds}
	
	In this section we prove the lower bounds of Theorem \ref{thm:main}.
	As usual, $\Om$ denotes an arbitrary open subset of $\R^n$.
	Let $\rho$ be a standard mollifier and let $\rho_{\delta}(x):=\delta^{-n} \rho(x/\delta)$, $\delta>0$.
	
	We start with the following lemma from
	Step 3 of the proof of Theorem 2.2  in \cite{FMS18}.
	\begin{lemma}\label{lem:mollifiers}
		Let $f\in L^1_{\loc}(\Om)$, let $\eps>0$, $\delta>0$, and suppose $\Om'\Subset \Om$
		with $\dist(\Om',\R^n\setminus\Om)>\delta$. Then
		\[
		\kappa_\eps(\rho_\delta*f,\Om')
		\leq  \kappa_\eps (f,\Om).
		\]
	\end{lemma}
	
	The following theorem gives the lower bound of \eqref{eq:Gamma p limit}.
	It was previously proved in \cite{ARBDN}, but we give a simpler proof.
	
	\begin{theorem}\label{thm:lower bound abs cont}
		Let $f\in L^1_{\loc}(\Om)$.
		Consider a family $\{f_\eps\}_{\eps>0}$ with $f_{\eps}\to f$ in $L^1_{\loc}(\Om)$ as $\eps\to 0$.
		Then we have
		\begin{align}\label{eq:lower bound}
			\tfrac{1}{4}|Df|(\Om)
			\le\liminf_{\eps\to 0}\kappa_\eps(f_\eps,\Om).
		\end{align}
	\end{theorem}
	\begin{proof}
		Fix open sets $\Om'\Subset \Om''\Subset  \Om$.
		Let $\delta>0$ with $\delta<\dist(\Om',\R^n\setminus \Om'')$.
		By \eqref{eq:SBV limit}, we have
		\begin{equation} \label{eq:f delta estimate}
			\begin{split}
				\tfrac{1}{4}	|D(\rho_\delta*f)|(\Om')
				&=\liminf\limits_{\eps\to 0} \kappa_\eps(\rho_\delta*f,\Om')\\
				&\leq \liminf\limits_{\eps\to 0}\kappa_\eps(\rho_\delta*f_\eps,\Om')
				+\limsup_{\eps\to 0}\kappa_\eps(\rho_\delta*(f-f_\eps),\Om')\\
				&\leq  \liminf\limits_{\eps\to 0} \kappa_\eps(f_\eps,\Om) 
				+\limsup_{\eps\to 0}\kappa_\eps(\rho_\delta*(f-f_\eps),\Om')
				\quad\textrm{by Lemma }\ref{lem:mollifiers}.
			\end{split}
		\end{equation}
		Take a family $\mathcal{F}_\eps$ of disjoint open cubes $Q\subset \Om'$ of side length $\eps$. 
		By Poincar\'e's inequality, for every $Q\in \mathcal{F}_\eps$ we have
		\begin{align*}
			\vint{Q}\Big|(\rho_\delta*(f-f_\eps))(x)-\vint{Q}(\rho_\delta*(f-f_\eps))(y)\,dy\Big|\,dx
			\leq C(n)  \eps^{1-n} \int_{Q}|(D\rho_\delta)*(f-f_\eps)|\,dx,
		\end{align*}
		where $C(n)$ is a constant depending only on $n$.
		Here
		\begin{align*}
			\sum_{Q\in \mathcal{F}_\eps}\int_{Q}|(D\rho_\delta)*(f-f_\eps)| \,dx
			&\leq \int_{\Om'}|(D\rho_\delta)*(f-f_\eps)| \,dx\\
			&\leq \Vert D\rho_\delta\Vert_{L^{1}(\R^n)}\Vert f-f_\eps\Vert_{L^1(\Om'')}\\
			&\to 0\quad\textrm{as }\eps\to 0,
		\end{align*}
		and so
		\[
		\limsup_{\eps\to 0}\kappa_\eps(\rho_\delta*(f-f_\eps),\Om')=0.
		\]
		Combining this with \eqref{eq:f delta estimate}, we get 
		\begin{align*}
			\tfrac{1}{4}	|D(\rho_\delta*f)|(\Om')\leq  \liminf\limits_{\eps\to 0} \kappa_\eps(f_\eps,\Om).
		\end{align*}
		Letting $\delta\to 0$ and then $\Om'\nearrow \Om$, we get \eqref{eq:lower bound}.
	\end{proof}
	
	The following proposition will be needed for the lower bound of \eqref{eq:Gamma infty limit}.
	
	\begin{proposition}\label{prop:lower bound jump}
		Let $\Om\subset \R^n$ be open and let $f\in L^{1}_{\loc}(\Om)$.
		Consider a family $\{f_\eps\}_{\eps>0}$ with  $f_{\eps}-f\in L^\infty_{\loc}(\Om)$ and $f_{\eps}-f\to 0$ in $L^\infty_{\loc}(\Om)$ as $\eps\to 0$.
		Then we have
		\[
			\tfrac{1}{2}|D^ j f|(\Om)
			\le\liminf_{\eps\to 0}\kappa_\eps(f_\eps,\Om).
		\]
	\end{proposition}
	\begin{proof}
		We can assume that $\liminf_{\eps\to 0}\kappa_\eps(f_\eps,\Om)<\infty$.
		Then by Theorem \ref{thm:lower bound abs cont}, we have $\Var(f,\Om)<\infty$.
		By choosing a suitable sequence $\eps_i\searrow 0$,
		we have $\kappa_{\eps_i}(f_{\eps_i},\Om)<\infty$ for every $i\in\N$ and
		\[
		\lim_{i\to\infty}\kappa_{\eps_i}(f_{\eps_i},\Om)
		=\liminf_{\eps\to 0}\kappa_\eps(f_\eps,\Om),
		\]
		and we need to show that
		\begin{equation}\label{eq:lower bound eps i}
			\tfrac{1}{2}|D^ j f|(\Om)\le \lim_{i\to\infty}\kappa_{\eps_i}(f_{\eps_i},\Om).
		\end{equation}
		For each $i\in\N$, we can choose an ``almost optimal'' collection $\mathcal F_{\eps_i}$
		of pairwise disjoint $\eps_i$-cubes in $\Om$, such that
		\begin{equation}\label{eq:almost optimality}
			\kappa_{\eps_i}(f_{\eps_i},\Om)-\frac{1}{i}\le \eps_i^{n-1}\sum_{Q\in \mathcal F_{\eps_i}}\vint{Q}\left|f_{\eps_i}-\vint{Q}f_{\eps_i}\,dy\right|dx\le \kappa_{\eps_i}(f_{\eps_i},\Om).
		\end{equation}
		Then we define the positive Radon measures
		\[
		\mu_{\eps_i}:=\eps_i^{n-1}\sum_{Q\in \mathcal F_{\eps_i}}\left|f_{\eps_i}-\vint{Q}f_{\eps_i}\,dy\right|\frac{\mathbbm{1}_Q}{|Q|}\,d\mathcal L^n.
		\]
		We have
		\[
		\mu_{\eps_i}(\Om)=\eps_i^{n-1}\sum_{Q\in \mathcal F_{\eps_i}}\vint{Q}\left|f_{\eps_i}-\vint{Q}f_{\eps_i}\,dy\right|dx
		\le \kappa_{\eps_i}(f_{\eps_i},\Om).
		\]
		Thus by passing to a subsequence (not relabeled), we have that
		$\mu_{\eps_i}\overset{*}{\rightharpoondown}\mu$ for some positive Radon measure $\mu$ on $\Om$.
		
		Fix $0<\sigma<1/10$.
		Recall that the set of approximate jump points $J_f\subset \Om$
		is a countably $\mathcal H^{n-1}$-rectifiable Borel set.
		Using Lusin's and Egorov's theorems,
		we find a Borel set $H\subset J_f$ such that $\mathcal H^{n-1}(H)<\infty$ and
		\begin{equation}\label{eq:H properties}
			|Df|(J_f\setminus H)<\sigma
			\quad\textrm{and}\quad
			\nu_f|_H,\,f^{+}|_H,\,f^{-}|_H
			\quad\textrm{are continuous,}
		\end{equation}
		and
		\begin{equation}\label{eq:uniform convergence}
			\vint{B_{\nu_f(x)}^{\pm}(x,s)}|f-f^{\pm}(x)|\,dy\to 0
			\quad\textrm{as }s\to 0
		\end{equation}
		uniformly for all $x\in H$.
		We can further assume that
		$H$ is the disjoint union of Borel sets $H_1,\ldots,H_m$ that are a positive distance
		from each other and from $\partial \Om$, and such that each $H_k$ is contained in a graph
		\[
		S_k:=\{x\in \R^n\colon h_k(\pi_k(x))=\pi_k^{\perp}(x)\},
		\]
		where $\pi_k$ is the orthogonal projection to a hyperplane $P_k$, and
		\begin{equation}\label{z1}
				h_k\in C^1(P_k)
			\quad\textrm{with}\quad
			\Vert \nabla h_k\Vert_{L^{\infty}(P_k)}\le \sigma.
		\end{equation}
		We can further assume that every point of $\pi_k(H_k)$ is a point of density $1$ on $P_k$,
		and that $\nu_f(x)$ is a normal vector of the graph $S_k$ for all
		$x\in H_k$.
		
		We consider $H_1$; the other $H_k$'s can be handled analogously.
		Fix a point $x\in H_1$.
		We can assume that
		$P_1=\R^{n-1}\times \{0\}$, which we simply denote by $\R^{n-1}$.
		For each $r>0$, consider $Q(x,r)$, that is,
		a cube centered at $x$ and with the sides parallel to coordinate axes.
		For all sufficiently small $r>0$, then choosing sufficiently small  $\eps>0$ we have
		\begin{equation}\label{eq:1 minus delta estimate}
		\left|	\pi_1(H_1)\cap Q^{n-1}(\pi_1(x),r-4\sqrt n\eps)\right|\ge (1-\sigma)r^{n-1},
		\end{equation}
		and (using \eqref{eq:H properties})
		\begin{equation}\label{eq:f continuity}
			|f^{\pm}(y)-f^{\pm}(x)|\le \sigma\quad\textrm{for all }y\in H_1\cap Q(x,r).
		\end{equation}
		For the moment, let us fix such a sufficiently small $r>0$ so that $Q(x,3r)$ does not intersect
		any $H_2,\ldots, H_m$ or $\partial\Om$,
		and also so that $\mu(\partial Q(x,r))=0$.
		For large enough $i\in\N$ such that
		\eqref{eq:1 minus delta estimate} holds with $\eps=\eps_i<r/(4n)$,
		consider a grid $\{Q^{n-1}(z_j,\eps_i)\}_{j=1}^N$ of $(n-1)$-cubes contained
		in $Q^{n-1}(\pi_1(x),r-2\sqrt n\eps_i)$ and  containing
		 $Q^{n-1}(\pi_1(x),r-4\sqrt n\eps_i)$
		and then consider the cubes $\{Q_j=Q(x_j,\eps_i)\}_{j=1}^N$ with $x_j:=(z_j,h(z_j))$.
		Denote by $Q_j$, $j=1,\ldots,M$ those cubes that intersect $H_1$.
		Thus for each $j=1,\ldots,M$, there is a point $y_j\in H_1\cap Q_j$.
		Denote the standard unit basis vectors of $\R^n$ by $e_1,\ldots,e_n$.
		Consider the half-spaces
		\[
		A_{j,1}:=\{y\in \R^n\colon (y-y_j)\cdot \nu_f(y_j)\ge 0\}
		\quad\textrm{and}\quad
		A_{j,2}:=\{y\in \R^n\colon (y-y_j)\cdot \nu_f(y_j)\le 0\}.
		\]
		Choosing $i$ large enough, by \eqref{eq:uniform convergence} we have for all $y\in H_1$
		\begin{equation}\label{eq:uniform convergence 2}
			\vint{B_{\nu_f(y)}^{\pm}(y,s)}|f-f^{\pm}(y)|\,dz\le \sigma
			\quad\textrm{for all }0<s\le 2n\eps_i.
		\end{equation}
		Thanks to \eqref{z1}, we also have
		\begin{equation}\label{eq:small intersection}
		\frac 12 -C\sigma\le \frac{|A_{j,1}\cap Q_j|}{|Q_j|}\le\frac 12 +C\sigma
		\end{equation}
		for some constant $C\ge 1$ depending only on $n$,
		and the same for $A_{j,2}$. We can assume that $C\sigma \le 1/4$.
		Then also
		\begin{equation}\label{eq:average estimate}
			\begin{split}
				&\left|f_{Q_j}-\frac 12 (f^{+}(y_j)+f^{-}(y_j))\right|\\
				&= \left|\frac{|A_{j,1}\cap Q_j|}{|Q_j|}f_{A_{j,1}\cap Q_j}
				+\frac{|A_{j,2}\cap Q_j|}{|Q_j|}f_{A_{j,2}\cap Q_j}-\frac 12 (f^{+}(y_j)+f^{-}(y_j))\right|\\
				&\le \frac{|A_{j,1}\cap Q_j|}{|Q_j|}\,\vint{A_{j,1}\cap Q_j}|f-f^{+}(y_j)|\,dy
				+\frac{|A_{j,2}\cap Q_j|}{|Q_j|}\,\vint{A_{j,2}\cap Q_j}|f-f^{-}(y_j)|\,dy\\
				&\qquad +\left|\frac{|A_{j,1}\cap Q_j|}{|Q_j|}-\frac 12\right||f^+(y_j)|
				+\left|\frac{|A_{j,2}\cap Q_j|}{|Q_j|}-\frac 12\right||f^-(y_j)|\\
				&\le (4n)^n\vint{B_{\nu_f(y_j)}^{+}(y_j,2n\eps_i)}|f-f^{+}(y_j)|\,dy
				+(4n)^n\vint{B_{\nu_f(y_j)}^{-}(y_j,2n\eps_i)}|f-f^{-}(y_j)|\,dy\\
				&\qquad +2C\sigma\sup_{H_1}(|f^+|+|f^-|)
				\quad\textrm{by }\eqref{eq:small intersection},
			\end{split}
		\end{equation}
		where by \eqref{eq:uniform convergence 2},
		\begin{equation}\label{eq:2 eps estimate}
			\begin{split}
				\vint{B_{\nu_f(y_j)}^{\pm}(y_j,2n\eps_i)}|f-f^{\pm}(y_j)|\,dy\le \sigma.
			\end{split}
		\end{equation}
		Using \eqref{eq:small intersection}, we estimate
		\begin{equation}\label{eq:initial estimate}
			\begin{split}
				\vint{Q_j}|f-f_{Q_j}|\,dy
				&\ge \frac {1-2C\sigma}{2} \vint{A_{j,1}\cap Q_j}|f-f_{Q_j}|\,dy
				+\frac {1-2C\sigma}{2} \vint{A_{j,2}\cap Q_j}|f-f_{Q_j}|\,dy.
			\end{split}
		\end{equation}
		Here
		\begin{align*}
			&\vint{A_{j,1}\cap Q_j}|f-f_{Q_j}|\,dy\\
			&\qquad= \vint{A_{j,1}
				\cap Q_j}|f-f^{+}(y_j)+\tfrac 12 f^{+}(y_j)-\tfrac 12 f^{-}(y_j)
			-f_{Q_j}+\tfrac 12 f^{+}(y_j)+\tfrac 12 f^{-}(y_j)|\,dy\\
			&\qquad\ge \frac 12 |f^{+}(y_j)-f^{-}(y_j)|-\vint{A_{j,1}
				\cap Q_j}|f-f^{+}(y_j)|\,dy-\left|f_{Q_j}-\frac 12 (f^{+}(y_j)+f^{-}(y_j))\right|\\
			&\qquad\ge \frac 12 |f^{+}(y_j)-f^{-}(y_j)|
			- 2\times (4n)^n\,\vint{B_{\nu_f(y_j)}^{+}(y_j,2n\eps_i)}|f-f^{+}(y_j)|\,dy\\
			&\qquad\qquad -(4n)^n\,\vint{B_{\nu_f(y_j)}^{-}(y_j,2n\eps_i)}|f-f^{-}(y_j)|\,dy
			-2C\sigma\sup_{H_1}(|f^+|+|f^-|)\quad\textrm{by }\eqref{eq:average estimate}\\
			&\qquad\ge \frac 12 |f^+(y_j)-f^-(y_j)|-(4n)^{n+1}\sigma
			-2C\sigma\sup_{H_1}(|f^+|+|f^-|)\quad\textrm{by }\eqref{eq:2 eps estimate}.
		\end{align*}
	Similarly, we also have 
		\begin{align*}
		&\vint{A_{j,2}\cap Q_j}|f-f_{Q_j}|\,dy\ge \frac 12 |f^+(y_j)-f^-(y_j)|-(4n)^{n+1}\sigma
		-2C\sigma\sup_{H_1}(|f^+|+|f^-|).
	\end{align*}
		Thus from \eqref{eq:initial estimate}, we get
		\begin{equation}\label{eq:cube estimate}
			\vint{Q_j}|f-f_{Q_j}|\,dy
			\ge \frac{1-2C\sigma}{2}|f^+(y_j)-f^-(y_j)|-(4n)^{n+1}\sigma
			-2C\sigma\sup_{H_1}(|f^+|+|f^-|).
		\end{equation}
		Note that $\{Q_j\}_{j=1}^N \cup \{Q\in \mathcal F_{\eps_i}\colon Q\setminus Q(x,r)\neq \emptyset\}$
		is a collection of disjoint $\eps_i$-cubes in $\Om$. Thus
		\begin{equation}\label{eq:replacing cubes}
			\begin{split}
				&\eps_{i}^{n-1}\sum_{Q\in\mathcal F_{\eps_i}\colon Q\setminus Q(x,r)= \emptyset}\,\vint{Q}|f_{\eps_i}-(f_{\eps_i})_{Q}|\,dy
				+\eps_{i}^{n-1}\sum_{Q\in\mathcal F_{\eps_i}\colon Q\setminus Q(x,r)\neq \emptyset}\,\vint{Q}|f_{\eps_i}-(f_{\eps_i})_{Q}|\,dy\\
				&=\eps_{i}^{n-1}\sum_{Q\in\mathcal F_{\eps_i}}\,\vint{Q}|f_{\eps_i}-(f_{\eps_i})_{Q}|\,dy\\
				&\ge \eps_{i}^{n-1}\sum_{j=1}^M\,\vint{Q_j}|f_{\eps_i}-(f_{\eps_i})_{Q}|\,dy
				+\eps_{i}^{n-1}\sum_{Q\in\mathcal F_{\eps_i}\colon Q\setminus Q(x,r)\neq \emptyset}\,\vint{Q}|f_{\eps_i}-(f_{\eps_i})_{Q}|\,dy
				-\frac{1}{i}
			\end{split}
		\end{equation}
		by \eqref{eq:almost optimality}.
		Then
		\begin{equation}\label{eq:inside Q cubes estimate}
		\begin{split}
			&\eps_{i}^{n-1}\sum_{Q\in\mathcal F_{\eps_i}\colon Q\setminus Q(x,r)= \emptyset}\,\vint{Q}|f_{\eps_i}-(f_{\eps_i})_{Q}|\,dy+\frac{1}{i}\\
			&\qquad\ge \eps_{i}^{n-1}\sum_{j=1}^M\,\vint{Q_j}|f_{\eps_i}-(f_{\eps_i})_{Q_j}|\,dy
			\quad\textrm{by }\eqref{eq:replacing cubes}\\
			&\qquad\ge \eps_i^{n-1}\sum_{j=1}^M\,\vint{Q_j}|f-f_{Q_j}|\,dy
			-2\eps_i^{n-1}M\Vert f-f_{\eps_i}\Vert_{L^{\infty}(Q(x,2r))}\\
			&\qquad\ge \eps_i^{n-1}\sum_{j=1}^M\left(\frac {1-2C\sigma}{2}|f^+(y_j)-f^-(y_j)|-(4n)^{n+1}\sigma
			-2C\sigma\sup_{H_1}(|f^+|+|f^-|)\right)\\
			&\qquad\qquad -2\eps_i^{n-1}M\Vert f-f_{\eps_i}\Vert_{L^{\infty}(Q(x,2r))}
			\quad\textrm{by }\eqref{eq:cube estimate}.
		\end{split}
	\end{equation}
		Note that $\eps_i^{n-1}M \le r^{n-1}$.
		Using \eqref{eq:f continuity}, we moreover get
		\begin{align*}
			&\eps_i^{n-1}\sum_{j=1}^M|f^+(y_j)-f^-(y_j)|\\
			&\qquad \ge |f^+(x)-f^-(x)|\left|\pi_1(H_1)\cap Q^{n-1}(\pi_1(x),r-4\sqrt n\eps_i)\right|
			-2\sigma M\eps_i^{n-1}\\
			&\qquad \ge |f^+(x)-f^-(x)|(1-\sigma)r^{n-1}
			-2\sigma r^{n-1}\quad\textrm{by }\eqref{eq:1 minus delta estimate}\\
			&\qquad \ge (1-\sigma)^2\int_{Q(x,r)\cap H_1}|f^+(x)-f^-(x)|\,d\mathcal H^{n-1}
			-2\sigma r^{n-1}\quad\textrm{since }\Vert \nabla h_1\Vert_{L^{\infty}}\le \sigma\\
			&\qquad \ge (1-\sigma)^2\int_{Q(x,r)\cap H}|f^+-f^-|\,d\mathcal H^{n-1}
			-4\sigma r^{n-1}\quad\textrm{by }\eqref{eq:f continuity}.
		\end{align*}
		Combining this with \eqref{eq:inside Q cubes estimate} and using again
		$\eps_i^{n-1}M \le r^{n-1}$, we have
		\begin{align*}
			&\eps_{i}^{n-1}\sum_{Q\in\mathcal F_{\eps_i}\colon Q\setminus Q(x,r)= \emptyset}\,\vint{Q}|f_{\eps_i}-(f_{\eps_i})_{Q}|\,dy\\
			&\quad \ge \frac {(1-2C\sigma)^3}{2}\int_{Q(x,r)\cap H}|f^+-f^-|\,d\mathcal H^{n-1}
			-(4n)^{n+2}\sigma r^{n-1}\\
			&\qquad-2C\sigma r^{n-1}\sup_{H_1}(|f^+|+|f^-|)
			-2r^{n-1}\Vert f-f_{\eps_i}\Vert_{L^{\infty}(Q(x,2r))}-\frac{1}{i}.
		\end{align*}
		Letting $i\to\infty$,  we get
		\begin{align*}
			&\liminf_{i\to\infty}\eps_{i}^{n-1}\sum_{Q\in\mathcal F_{\eps_i}\colon Q\setminus Q(x,r)= \emptyset}\,\vint{Q}|f_{\eps_i}-(f_{\eps_i})_{Q}|\,dy\\
			& \ge \frac {(1-2C\sigma)^3}{2}\int_{Q(x,r)\cap H}|f^+-f^-|\,d\mathcal H^{n-1}
			-(4n)^{n+2}\sigma r^{n-1}
			-2C\sigma r^{n-1}\sup_{H_1}(|f^+|+|f^-|).
		\end{align*}
		Since $\mu(\partial Q(x,r))=0$, by the weak* convergence 	
		$\mu_{\eps_i}\overset{*}{\rightharpoondown}\mu$ we get
		\begin{align*}
			&\mu(Q(x,r))\\
			&\ \  =\lim_{i\to\infty}\mu_{\eps_i}(Q(x,r))\\
			&\ \  \ge \liminf_{i\to\infty}\eps_{i}^{n-1}\sum_{Q\in\mathcal F_{\eps_i}\colon Q\setminus Q(x,r)= \emptyset}\,\vint{Q}|f_{\eps_i}-(f_{\eps_i})_{Q}|\,dy
			\quad\textrm{by definition of }\mu_{\eps_i}\\
			&\ \  \ge \frac {(1-2C\sigma)^3}{2}\int_{Q(x,r)\cap H}|f^+-f^-|\,d\mathcal H^{n-1}
			-(4n)^{n+2}\sigma r^{n-1}
			-2C\sigma r^{n-1}\sup_{H_1}(|f^+|+|f^-|).
		\end{align*}
		Recall the representation \eqref{eq:jump part representation}.
		Using \eqref{eq:diff wrt cubes}, this proves that
		\begin{align*}
			\frac{d\mu}{d|D^jf|}(x)
			&= \liminf_{r\to 0}\frac{\mu(Q(x,r))}{|D^jf|(Q(x,r))}\\
			& =\liminf_{r\to 0}\frac{\mu(Q(x,r))}{\int_{Q(x,r)\cap H}|f^+-f^-|\,d\mathcal H^{n-1}}\\
			&\ge \frac {(1-2C\sigma)^3}{2}-\frac{(4n)^{n+2}\sigma+
				2C\sigma \sup_{H_1}(|f^+|+|f^-|)}{|f^+(x)-f^-(x)|}
		\end{align*}
		for $|D^jf|$-a.e. $x\in H_1$,
		and so in total for $|D^jf|$-a.e. $x\in H$.
		Since we have $|Df|(J_f\setminus H)<\sigma$ with $\sigma>0$
		arbitrarily small, letting $\sigma\to 0$ we get
		\[
		\frac{d\mu}{d|D^jf|}(x)
		\ge \frac 12\quad\textrm{for }|D^jf|\textrm{-a.e. }x\in \Om.
		\]
		By \eqref{eq:Radon Nikodym}, we get
		$\mu\ge \frac{1}{2}|D^jf|$ in $\Om$.
		By lower semicontinuity with respect to weak* convergence, we now get
		\begin{align*}
			\lim_{i\to\infty}\kappa_{\eps_i}(f_{\eps_i},\Om)
			= \liminf_{i\to\infty}\mu_{\eps_i}(\Om)
			\ge \mu(\Om)
			\ge \tfrac{1}{2} |D^jf|(\Om),
		\end{align*}
		proving \eqref{eq:lower bound eps i}.
	\end{proof}
	
	\begin{remark}
		Note that as a corollary, we can consider the case $f_{\eps}\equiv f$,
		giving one part of the lower bound of \eqref{eq:SBV limit}.
		In our proof, the measures $\mu_{\eps_j}$ are defined as in \cite{ARBDN},
		whereas our method of defining the set $H$ is inspired by \cite{FMS16,FMS18}.
	\end{remark}
	
	\begin{proof}[Proof of the lower bound of \eqref{eq:Gamma infty limit}]
		Let $f\in L^1_{\loc}(\Om)$ and take a family $\{f_{\eps}\}_{\eps>0}$ such that
		$f_\eps-f\to 0$ in $L^{\infty}_{\loc}(\Om)$.
		We can assume that $\liminf_{\eps\to 0}\kappa_\eps(f_\eps,\Om)<\infty$.
		Then by Theorem \ref{thm:lower bound abs cont}, we have $\Var(f,\Om)<\infty$.
		
		Fix $\delta>0$.
		We find a compact set $K\subset J_f$ such that $|Df|(J_f\setminus K)<\delta$.
		We also find an open set $W$ with $K\subset \overline{W}\subset \Om$ and $|Df|(\overline{W}\setminus K)<\delta$.
		Then we have
		\begin{align*}
			\liminf_{\eps\to 0}\kappa_{\eps}(f_{\eps},\Om)
			&\ge \liminf_{\eps\to 0}\kappa_{\eps}(f_{\eps},W)
			+\liminf_{\eps\to 0}\kappa_{\eps}(f_{\eps},\Om\setminus \overline{W})\\
			&\ge \tfrac 12 |D^j f|(W)+ \tfrac 14 |D f|(\Om\setminus \overline{W})\quad \textrm{by Prop. }\ref{prop:lower bound jump}\textrm{ and Thm. }\ref{thm:lower bound abs cont}\\
			&\ge \tfrac 12 |D^j f|(J_f)-\delta+ \tfrac 14 |D f|(\Om\setminus J_f)-\delta
			\\
			&=\tfrac 12 |D^j f|(\Om)+ \tfrac 14 |D^a f|(\Om)+\tfrac 14 |D^c f|(\Om)-2\delta.
		\end{align*}
		Since $\delta>0$ was arbitrary, we have the result.
	\end{proof}
	
	\section{Upper bounds}
	
	In this section we prove the upper bounds of Theorem \ref{thm:main}.
	As usual, $\Om$ denotes an arbitrary open subset of $\R^n$.
	
	The following approximation result for BV functions in the $L^p$-norm is straightforward to prove,
	see \cite[Section 3]{LLW}.
	
	\begin{theorem}\label{thm:BV Lp approximation result}
		Let $f\in L^1_{\loc}(\Om)$ with $\Var(f,\Om)<\infty$, and let $1\le p<\infty$.
		Then there exists a sequence $\{f_i\}_{i=1}^{\infty}\subset W^{1,1}_{\loc}(\Om)$ such that
		$f_i-f \to 0$ in $L^1(\Om)\cap L^p(\Om)$ and
		\[
		\int_{\Om}|\nabla f_i|\,dx\to |Df|(\Om).
		\]
	\end{theorem}
	
	\begin{proof}[Proof of the upper bound of \eqref{eq:Gamma p limit}]
		Fix an arbitrary $1\le p<\infty$.
		Let $f\in L^1_{\loc}(\Om)$. We can assume that $\Var(f,\Om)<\infty$.
		Applying Theorem \ref{thm:BV Lp approximation result}, we find a sequence of functions
		$\{f_i\}_{i=1}^{\infty}\subset W^{1,1}_{\loc}(\Om)$ such that
		$f_i-f \to 0$ in $L^p(\Om)$ and
		\[
		\int_{\Om}|\nabla f_i|\,dx\le |Df|(\Om)+1/i.
		\]
		For every $i\in\N$, by \eqref{eq:SBV limit}
		we find $0<\eps_i<1/i$ such that 
		\[
		\kappa_{\eps}(f_i,\Om)\le \frac 14 \int_{\Om}|\nabla f_i|\,dx+1/i
		\]
		for all $0<\eps\le \eps_i$.
		We can also assume that $\eps_{i+1}<\eps_i$.
		Define $f_{\eps}:=f_i$ for all $\eps\in (\eps_{i+1},\eps_i]$.
		Then for all $\eps\in (\eps_{i+1},\eps_i]$, we have
		\begin{align*}
			\kappa_{\eps}(f_\eps,\Om)
			&= \kappa_{\eps}(f_i,\Om) \\
			&\le \frac 14 \int_{\Om}|\nabla f_i|\,dx+1/i\\
			&\le \tfrac 14 |Df|(\Om)+2/i.
		\end{align*}
		Note that $i\to \infty$ as $\eps\to 0$, and so we get
		\[
		\limsup_{\eps\to 0}\kappa_{\eps}(f_\eps,\Om)\le \tfrac 14 |Df|(\Om).
		\]
	\end{proof}
	
	In the construction of the standard Cantor-Vitali function $f$ on $[0,1]$, note that the approximating
	Lipschitz functions converge \emph{uniformly} to $f$.
	To prove the upper bound for the $\Gamma_{\infty}$-convergence, we need the following
	approximation result which has a complicated proof given in \cite{L-appr}.
	
	\begin{theorem}\label{thm:approximation result}
		Let $f\in\BV(\Om)$. Then there exists a
		sequence $\{f_i\}_{i=1}^{\infty}\subset \SBV(\Om)$
		such that $f_i-f\to 0$ in $L^{\infty}(\Om)$,
		\[
		\limsup_{i\to \infty}|D^a f_i|(\Om)\le  |D^a f|(\Om)+|D^c f|(\Om),
		\quad\textrm{and}\quad
		\limsup_{i\to \infty}|D^j f_i|(\Om)\le  |D^j f|(\Om).
		\]
	\end{theorem}
	\begin{proof}
		Let $\delta>0$.
		From  \cite[Proposition 5.3]{L-appr} we get a function $h\in \SBV(\Om)$
		such that $\Vert h-f\Vert_{L^{\infty}(\Om)}<\delta$,
		$|Dh|(\Om)<|Df|(\Om)+\delta$, and $\mathcal H^{n-1}(J_h\setminus J_f)=0$.
		The proof of \cite[Proposition 5.3]{L-appr} additionally contains the following facts:
		for a set
		\[
		S:=\{x\in J_f\colon f^+(x)-f^-(x)\ge \delta'\},
		\]
		where we can choose $\delta'>0$
		arbitrarily small,
		we have $\mathcal H^{n-1}(J_h\setminus S)=0$ and $|D(h-f)|(S)<\delta$.
		Recall the BV theory described at the end of Section \ref{sec:definitions}.
		Choosing $\delta'$ sufficiently small, we also have $|D^jf|(\Om\setminus S)<\delta$, and so
		in total we get
		\[
		|D^j(h-f)|(\Om)=|D^j(h-f)|(J_f)
		=|D^j(h-f)|(S)+|D^j f|(J_f\setminus S)
		<2\delta.
		\]
		Thus  $|D^j h|(\Om)\le |D^j(h-f)|(\Om)+|D^j f|(\Om)\le |D^j f|(\Om)+2\delta$, and
		\begin{align*}
			|D^a h|(\Om)
			&=|D h|(\Om)-|D^j h|(\Om)\\
			&\le |Df|(\Om)+\delta-|D^j f|(\Om)+|D^j(h-f)|(\Om)\\
			&\le |Df|(\Om)+\delta-|D^j f|(\Om)+2\delta\\
			&=|D^a f|(\Om)+|D^c f|(\Om)+3\delta.
		\end{align*}
		Now, choosing $\delta=1/i$ and defining $f_i$ to be the corresponding $h$ for each $i\in\N$,
		we obtain the desired sequence.
	\end{proof}
	
	By using a rather simple partition of unity argument inside $\Om$,
	see \cite[Lemma 3.2]{L-appr} and its proof,
	we can get the following variant of the above theorem.
	
	\begin{corollary}\label{cor:approximation result}
		Let $f\in L^1_{\loc}(\Om)$
		with $\Var(f,\Om)<\infty$. Then there exists a
		sequence $\{f_i\}_{i=1}^{\infty}\subset \SBV_{\loc}(\Om)$
		such that $f_i-f\to 0$ in $L^{\infty}(\Om)$,
		\[
		\limsup_{i\to \infty}|D^a f_i|(\Om)\le  |D^a f|(\Om)+|D^c f|(\Om),
		\quad\textrm{and}\quad
		\limsup_{i\to \infty}|D^j f_i|(\Om)\le  |D^j f|(\Om).
		\]
	\end{corollary}
	
	\begin{proof}[Proof of the upper bound of \eqref{eq:Gamma infty limit}]
		Let $f\in L^1_{\loc}(\Om)$. We can assume that $\Var(f,\Om)<\infty$.
		Using Corollary \ref{cor:approximation result}, choose a sequence of functions
		sequence $\{f_i\}_{i=1}^{\infty}\subset \SBV_{\loc}(\Om)$
		such that $f_i-f\to 0$ in $L^{\infty}(\Om)$,
		\[
		|D^a f_i|(\Om)\le  |D^a f|(\Om)+|D^c f|(\Om)+1/i,
		\quad\textrm{and}\quad
		|D^j f_i|(\Om)\le  |D^j f|(\Om)+1/i.
		\]
	Thanks to \eqref{eq:SBV limit}, 	for every $i\in\N$
		we find $0<\eps_i<1/i$ such that 
		\[
		\kappa_{\eps}(f_i,\Om)\le \tfrac 14 |D^a f_i|(\Om)
		+\tfrac 12 |D^j f_i|(\Om)+1/i
		\]
		for all $0<\eps\le \eps_i$.
		We can also assume that $\eps_{i+1}<\eps_i$.
		Define $f_{\eps}:=f_i$ for all $\eps\in (\eps_{i+1},\eps_i]$.
		Then for all $\eps\in (\eps_{i+1},\eps_i]$, we have
		\begin{align*}
			\kappa_{\eps}(f_\eps,\Om)
			&= \kappa_{\eps}(f_i,\Om) \\
			&\le \tfrac 14 |D^a f_i|(\Om)+\tfrac 12 |D^j f_i|(\Om)+1/i\\
			&\le \tfrac 14|D^af|(\Om)+\tfrac 14|D^cf|(\Om)+\tfrac 12 |D^j f|(\Om)+2/i.
		\end{align*}
		Note that $i\to \infty$ as $\eps\to 0$, and so we get
		\[
		\limsup_{\eps\to 0}	\kappa_{\eps}(f_\eps,\Om)
		\le \tfrac 14|D^af|(\Om)+\tfrac 14|D^cf|(\Om)+\tfrac 12 |D^j f|(\Om).
		\]
	\end{proof}

\end{document}